\documentclass[12pt]{article}

\usepackage{amssymb,tikz,float,enumerate}
\usepackage{amsthm,amsfonts, amsbsy, amssymb,amsmath,graphicx}
\usepackage{graphics}
\usepackage[english]{babel}
\usepackage{graphicx}
\usepackage{enumerate}
\usepackage{tikz}
\usepackage{hyperref}
\usepackage{color}
\usepackage[T1]{fontenc}
\usepackage[margin=0.75in]{geometry}
\usepackage[ruled,vlined]{algorithm2e}

\newtheorem{theorem}{\bf Theorem}[section]

\newtheorem{corollary}[theorem]{\bf Corollary}
\newtheorem{lemma}[theorem]{\bf Lemma}
\newtheorem{proposition}[theorem]{\bf Proposition}

\newtheorem{problem}[theorem]{\bf Problem}

\title{Grundy Packing Coloring of Graphs\thanks{This work was partially supported by the Scientific and Technological Research Council of Turkey (TÜBİTAK) by grant no. 124F114 and by TUBITAK 2221 Fellowships for Visiting Scientists and Scientitsts on Sabbatical Leave program. The second author was partially supported by the Slovenian Research and Innovation Agency by program No. P1-0297}}

\author{{{Didem Gözüpek}$^{1}$  {Iztok Peterin}}$^{2}$\\
\small Department of Computer Engineering, Gebze Technical University, Türkiye$^1$\\
\small Faculty of Electrical Engineering and Computer Science, University of Maribor, Slovenia$^2$\\
\small Institute of Mathematics, Physics and Mechanics, Ljubljana,  Slovenia$^2$\\
{\small Emails: \em didem.gozupek@gtu.edu.tr}$^1$; {\small\em iztok.peterin@um.si}$^2$}

\date{}

\begin{document}

\maketitle

\begin{abstract}
A map $c:V(G)\rightarrow\{1,\dots,k\}$ of a graph $G$ is a packing $k$-coloring if every two different vertices of the same color $i\in \{1,\dots,k\}$ are at distance more than $i$. The packing chromatic number $\chi_{\rho}(G)$ of $G$ is the smallest integer $k$ such that there exists a packing $k$-coloring. In this paper we introduce the notion of \textit{Grundy packing chromatic number}, analogous to the Grundy chromatic number of a graph. We first present a polynomial-time algorithm that is based on a greedy approach and gives a packing coloring of \textcolor{black}{any graph} $G$. We then define the Grundy packing chromatic number $\Gamma_{\rho}(G)$ of a graph $G$ as the maximum value that \textcolor{black}{this algorithm yields in $G$}. We present several properties of $\Gamma_{\rho}(G)$, provide results on the complexity of the problem as well as bounds and some exact results for $\Gamma_{\rho}(G)$.  
\end{abstract}

{\it Keywords: Grundy packing chromatic number, packing coloring, Grundy packing coloring, algorithm, \textcolor{black}{independent domination} number, diameter} 

{\it AMS Subject Classification (2020): 05C15, 05C12, 05C69}

\section{Introduction and preliminaries} 

The concept of packing chromatic number $\chi_{\rho}(G)$ of a graph $G$ was initiated by Goddard et. al. \cite{GHHHR} under the name \emph{broadcast chromatic number} in 2008. The name was changed by Bre\v sar et al. \cite{bresar-klavzar-rall-1} in the next publication on this topic (officially published a year earlier in 2007) because color classes of a packing coloring of $G$ form packings of $G$. First, the focus was on different infinite latices such as (2- and 3-dimensional) grid, hexagonal lattice and triangular lattice. For hexagonal lattice ${\cal H}$, preliminary work on ${\cal H}$ was done already in \cite{bresar-klavzar-rall-1} and by Fiala et al. \cite{fiala-klavzar-lidicky}. Later, Kor\v ze and Vesel \cite{KoVe1} showed that $\chi_{\rho}({\cal H})=7$. Packing chromatic number is infinite for some latices like for $\mathbb{Z}^3$ as shown by Finbow and Rall \cite{finbow-rall}. Later contributions on $\chi_{\rho}(G)$ are oriented into graph operations like Cartesian product \cite{jacobs-jonck-joubert,KoVe1,NaST} or lexicographic product \cite{BoPe}, some special graph classes \cite{ArNT,BaKL,BiGK,BrKR,GaHT,KoVe2} and to packing chromatic critical graphs \cite{BrFe}. Gastineau et al. \cite{GaKT} introduced an $S$-packing coloring that is a generalization of packing coloring which also draw some attention in the community, mainly as $(d,h)$-packing coloring. For more on packing colorings and packing chromatic number we recommend the survey \cite{BFKR}.

For graphs with diameter two, a nice connection between the packing chromatic number $\chi_{p}(G)$ and independence number $\alpha (G)$ of $G$  was shown in \cite{GHHHR} as follows:
\begin{equation}\label{diam2}
\chi_{\rho}(G)=|V(G)|-\alpha (G)+1.
\end{equation}
This result implies that the packing coloring problem is NP-hard even for graphs with diameter two, which is usual for different coloring problems as well. Later, Fiala and Golovach \cite{FiGo} showed that the existence of a packing $k$-coloring is NP-complete also for trees. Despite that fact no (polynomial-time) algorithm was presented to construct a packing coloring of an arbitrary graph $G$, in \cite{KoVe1,KoVe2} the authors did use a computer (SAT solver) to construct packing colorings, but an explicit algorithm was not presented. Kor\v ze et al. \cite{KoMV} presented a heuristic algorithm for the more general problem of $(d,n)$-packing coloring for several (infinite) latices, but not for arbitrary graphs. 

In this paper, we fill this gap in the literature by presenting a polynomial-time algorithm that produces a packing coloring for an arbitrary graph $G$. It is based on a famous greedy algorithm, sometimes also called first-fit algorithm, with some additional property that prevents vertices of the same color to be too close to each other in $G$. As usual for a greedy algorithm, our algorithm also gives only an upper bound for $\chi_{\rho}(G)$. This motivates us to introduce the study of the worst possible case of our algorithm and the number $\Gamma_{\rho}(G)$ of colors in such a case called \emph{Grundy packing chromatic number} as an analogue to the Grundy number $\Gamma(G)$ that represents the maximum number of colors that can be produced by a greedy \textcolor{black}{proper coloring} algorithm for $G$. 

The term \textit{Grundy} was initially coined because \cite{Grun}. 
The \textit{Grundy number} or\textit{ Grundy chromatic number} of an undirected graph is the maximum number of colors that can be used by a greedy coloring strategy that considers the vertices of the graph in a sequence and assigns each vertex its first available color, using a vertex ordering chosen to use as many colors as possible  \cite{CFFMMR}. One of the earliest uses of the term \textit{Grundy number}, as we recognize it today, dates back to Erd\" os et al. \cite{Erdo} where the equality between $\Gamma(G)$ and certain chromatic number of graphs is presented. The complexity of the problem is studied by Bonnet et al. in \cite{BFKS}. Several general bounds for $\Gamma(G)$ are presented by Zaker \cite{Zake}. Grundy number of graph products are studied by Ast\' e et al. \cite{AHLS} and Campos et al. \cite{CGHLSM}. A study on a special graph class due to Araujo and Linhares-Sales is in \cite{ArLS}. Recently, several papers have compared $\Gamma(G)$ with the b-chromatic number \cite{HaSa, MaZa1, MaZa2, MaZa3}. Interesting upper bounds for $\Gamma(G)$ in terms of the domination number and girth of a graph are attributed to Khaleghi and Zaher \cite{KhZa}. We conclude this brief overview by noting that the term 'Grundy' is also employed in domination theory. Interested readers are referred to the seminal paper \cite{BBGK} on this subject for further insight. 

We now introduce the notation used throughout the rest of this paper. Let $G$ be a graph without loops and multiple edges. The  number of edges on a shortest path between vertices $u$ and $v$ of $G$ is called the \emph{distance} $d(u,v)$ between $u$ and $v$. The maximum length of a shortest path between $v\in V(G)$ and any other vertex $u$ is called the \emph{eccentricity} of $v$, denoted by ${\rm ecc}(v)$. The maximum eccentricity among all vertices is the \emph{diameter} of $G$, denoted by ${\rm diam}(G)$, and the minimum eccentricity is the \emph{radius} of $G$, denoted by ${\rm rad}(G)$. A vertex of minimum eccentricity is called a \emph{central vertex} of $G$. If $d(u,v)={\rm diam}(G)$, then $u$ (and $v$) is referred to as a \emph{diametrical vertex}, while and $u$ and $v$ form a \emph{diametrical pair}. A vertex is called \emph{universal} if it is adjacent to all the other vertices. We denote open and closed neighborhoods of $v\in V(G)$ by $N(v)$ and $N[v]$, respectively. A complete graph, a path and a cycle on $n$ vertices are denoted by $K_n$, $P_n$ and $C_n$, respectively. A \emph{clique} is a subset of vertices $V(G)$ of graph $G$ such that every two distinct vertices in the clique are adjacent.

Let $D\subseteq V(G)$. If no two vertices from $D$ are adjacent, then $D$ is an \emph{independent set}. An independent set that is not a subset of any other independent set is called \textit{maximal}. The maximum cardinality of an independent set is called the \emph{independence number} $\alpha(G)$ of $G$, whereas the minimum cardinality of a maximal independent set of $G$ is called the \emph{independent domination number} $i(G)$ of $G$. A maximal independent set of cardinality $\alpha (G)$ or $i(G)$ is simply called an $\alpha(G)$-set or $i(G)$-set, respectively.

Let $t$ be a positive integer. A set $X\subseteq V(G)$ is a $t$-\emph{packing} if any two distinct vertices from $X$ are at a distance greater than $t$. The maximum cardinality of a $t$-packing of $G$ is called the $t$-\emph{packing number} of $G$ and is denoted by $\rho_t(G)$. Notice, that if $t=1$, then the $1$-packing number is equal to $\alpha (G)$. A \emph{packing} $k$-\emph{coloring} of $G$ is a map $c: V(G) \rightarrow \{1,\dots,k\}$, such that if $c(u)=c(v)=i$, then $d_G(u,v)>i$. Hence, an $i^{th}$ color class of a packing coloring represents an $i$-packing of $G$. Clearly, $\chi_{\rho}(G)$ is the minimum integer $k$ for which there exists a packing $k$-coloring of $G$. \textcolor{black}{Observe feasible packing colorings of paths $P_4,P_5,P_8$ and $P_{15}$ in Figure \ref{paths}, which will play a role in the last section.}

\begin{figure}[H]
\begin{center}
\begin{tikzpicture}[xscale=.8, yscale=.8, style=thick,x=1cm,y=1cm]
\def\vr{3.5pt} 


\path (-0.5,3) coordinate (a1);
\path (1,3) coordinate (a2);
\path (2.5,3) coordinate (a3);
\path (4,3) coordinate (a4);
\path (7,3) coordinate (b1);
\path (8.5,3) coordinate (b2);
\path (10,3) coordinate (b3);
\path (11.5,3) coordinate (b4);
\path (13,3) coordinate (b5);
\path (1,1.5) coordinate (c1);
\path (2.5,1.5) coordinate (c2);
\path (4,1.5) coordinate (c3);
\path (5.5,1.5) coordinate (c4);
\path (7,1.5) coordinate (c5);
\path (8.5,1.5) coordinate (c6);
\path (10,1.5) coordinate (c7);
\path (11.5,1.5) coordinate (c8);
\path (-1.75,0) coordinate (d2);
\path (-.5,0) coordinate (d3);
\path (0.75,0) coordinate (d4);
\path (2,0) coordinate (d5);
\path (3.25,0) coordinate (d6);
\path (4.5,0) coordinate (d7);
\path (5.75,0) coordinate (d8);
\path (7,0) coordinate (d9);
\path (8.25,0) coordinate (d10);
\path (9.5,0) coordinate (d11);
\path (10.75,0) coordinate (d12);
\path (12,0) coordinate (d13);
\path (13.25,0) coordinate (d14);
\path (14.5,0) coordinate (d15);
\path (15.75,0) coordinate (d16);
\draw (c1) -- (c8);
\draw (a1) -- (a4);
\draw (b1) -- (b5);
\draw (d2) -- (d16);

\draw (a1) [fill=white] circle (\vr); \draw (a2) [fill=white] circle (\vr);
\draw (a3) [fill=white] circle (\vr); \draw (a4) [fill=white] circle (\vr);
\draw (b1) [fill=white] circle (\vr); \draw (b2) [fill=white] circle (\vr);
\draw (b3) [fill=white] circle (\vr); \draw (b4) [fill=white] circle (\vr);
\draw (b5) [fill=white] circle (\vr); \draw (c1) [fill=white] circle (\vr);
\draw (c2) [fill=white] circle (\vr); \draw (c3) [fill=white] circle (\vr);
\draw (c4) [fill=white] circle (\vr); \draw (c5) [fill=white] circle (\vr);
\draw (c6) [fill=white] circle (\vr); \draw (c7) [fill=white] circle (\vr);
\draw (c8) [fill=white] circle (\vr); 
\draw (d2) [fill=white] circle (\vr); \draw (d3) [fill=white] circle (\vr);
\draw (d4) [fill=white] circle (\vr); \draw (d5) [fill=white] circle (\vr);
\draw (d6) [fill=white] circle (\vr); \draw (d7) [fill=white] circle (\vr);
\draw (d8) [fill=white] circle (\vr); \draw (d9) [fill=white] circle (\vr);
\draw (d10) [fill=white] circle (\vr); \draw (d11) [fill=white] circle (\vr);
\draw (d12) [fill=white] circle (\vr); \draw (d13) [fill=white] circle (\vr);
\draw (d14) [fill=white] circle (\vr); \draw (d15) [fill=white] circle (\vr);
\draw (d16) [fill=white] circle (\vr); 

\draw[anchor = north] (a1) node {$1$};
\draw[anchor = north] (a2) node {$2$};
\draw[anchor = north] (a3) node {$1$};
\draw[anchor = north] (a4) node {$3$};
\draw[anchor = north] (b1) node {$1$};
\draw[anchor = north] (b2) node {$3$};
\draw[anchor = north] (b3) node {$2$};
\draw[anchor = north] (b4) node {$1$};
\draw[anchor = north] (b5) node {$4$};
\draw[anchor = north] (c1) node {$1$};
\draw[anchor = north] (c2) node {$4$};
\draw[anchor = north] (c3) node {$2$};
\draw[anchor = north] (c4) node {$1$};
\draw[anchor = north] (c5) node {$3$};
\draw[anchor = north] (c6) node {$5$};
\draw[anchor = north] (c7) node {$1$};
\draw[anchor = north] (c8) node {$2$};
\draw[anchor = north] (d2) node {$2$};
\draw[anchor = north] (d3) node {$1$};
\draw[anchor = north] (d4) node {$4$};
\draw[anchor = north] (d5) node {$3$};
\draw[anchor = north] (d6) node {$1$};
\draw[anchor = north] (d7) node {$2$};
\draw[anchor = north] (d8) node {$5$};
\draw[anchor = north] (d9) node {$1$};
\draw[anchor = north] (d10) node {$6$};
\draw[anchor = north] (d11) node {$2$};
\draw[anchor = north] (d12) node {$1$};
\draw[anchor = north] (d13) node {$3$};
\draw[anchor = north] (d14) node {$4$};
\draw[anchor = north] (d15) node {$1$};
\draw[anchor = north] (d16) node {$2$};

\draw(-1.2,3) node {$P_4$:};
\draw(6.3,3) node {$P_5$:};
\draw(0.3,1.5) node {$P_8$:};
\draw(-2.5,0) node {$P_{15}$:};
\end{tikzpicture}
\end{center}
\caption{Some packing colorings of selected paths.}
\label{paths}
\end{figure}

The remainder of this paper is organized as follows. In Section \ref{sec:Algorithm}, we provide a greedy polynomial-time algorithm that finds a packing coloring of a graph. In Section \ref{sec:GrundyPackingChromatic}, we define the Grundy packing chromatic number $\Gamma_{\rho}(G)$ of a graph $G$ and present a polynomial-time construction for $G$ that connects $\Gamma_{\rho}(G)$ with the \textcolor{black}{independent domination number,} which yields the complexity of finding $\Gamma_{\rho}(G)$. In Section \ref{sec:LargePackingGrundy} we describe all graphs with $\Gamma_{\rho}(G)\in\{|V(G)|,|V(G)|-1\}$. Section \ref{sec:SmallDiam} is devoted to $\Gamma_{\rho}(G)$ for graphs with small diameter (at most three). We conclude the paper with Section \ref{sec:Conclusion}, where several additional aspects of $\Gamma_{\rho}(G)$ are discussed and several open problems are presented.

\section{Greedy Packing Coloring Algorithm} \label{sec:Algorithm}
Recall that the greedy algorithm employed in Grundy coloring is such that each time an uncolored vertex $v$ is picked, it is colored with the smallest positive integer not present in the neighborhood of $v$. We adapt this greedy algorithm for standard graph coloring for packing colorings. To this end, we additionally keep track of the distances of already colored vertices in the following way: an $n$-tuple $(v_1,\dots,v_n)$  is constructed for each vertex $v \in V(G)$, where $n$ is the number of vertices and $v_i$ is a binary variable indicating whether color $i$ is available for vertex $v$. More precisely, for an uncolored vertex $v$ at some stage of our algorithm, $v_i=0$ means that color $i$ was already assigned to a vertex at distance at most $i$ to $v$ and hence $v$ cannot be colored with $i$. Likewise, $v_i=1$ means that color $i$ is available for vertex $v$. So, when $v$ is processed by our greedy algorithm, we simply need to find the smallest index $i\in\{1,\dots,n\}$ such that $v_i=1$ and $v_1=\cdots=v_{i-1}=0$, and then  $v$ is colored with $i$. To ensure that the obtained coloring is a feasible packing coloring, that is, the vertices of the same color are within distance at least $i+1$ apart, we run a BFS algorithm starting from the vertex under consideration, say $v$ with $c(v)=i$, and set $u_i=0$ for every vertex $u$ in any of first $i$ distance levels of BFS algorithm. The pseudocode of our algorithm is as follows. \bigskip

\begin{algorithm}[H]\label{alg:greedy}
\SetAlgoLined
\KwData{Graph $G$ and every vertex with $|V(G)|$ dimensional array with $1$s for every vertex.}
\KwResult{Packing coloring $c$ of $G$.}
\While{there exists an uncolored vertex}{
	1. pick an uncolored vertex $v$\\
	2. find first non-zero entry $i$ in array for $v$ and set $c(v)=i$\\
	3. find the vertices that are at distance at most $i$ from $v$ using BFS algorithm and for every such uncolored vertex $u$ set  $u_i=0$}
\caption{Packing coloring}
\end{algorithm}\bigskip

We illustrate a step-by-step execution of Algorithm \ref{alg:greedy} on an example graph in Figure \ref{exalg}, where Step 1 depicts the initial condition and at each step the vertex shown in black is colored and only the color numbers are displayed for the vertices previously colored.

\begin{figure}[htb]
\begin{center}
\begin{tikzpicture}[xscale=.8, yscale=.8, style=thick,x=1cm,y=1cm]
\def\vr{3.5pt} 


\path (0,0) coordinate (a);
\path (1.5,0) coordinate (b);
\path (2.5,1) coordinate (c);
\path (2.5,-1) coordinate (d);
\path (3.5,0) coordinate (e);
\path (5,0) coordinate (f);


\draw (a) -- (b) -- (c) -- (e) -- (f);
\draw (b) -- (d) -- (e);
\draw (c) -- (d);

\draw (a) [fill=white] circle (\vr);
\draw (b) [fill=white] circle (\vr);
\draw (c) [fill=white] circle (\vr);
\draw (d) [fill=white] circle (\vr);
\draw (e) [fill=white] circle (\vr);
\draw (f) [fill=white] circle (\vr);

\draw(-1,-0.5) node {$(1,1,1,1,1,1)$};
\draw(0,0.5) node {$(1,1,1,1,1,1)$};
\draw[anchor = south] (c) node {$(1,1,1,1,1,1)$};
\draw[anchor = north] (d) node {$(1,1,1,1,1,1)$};
\draw(5,0.5) node {$(1,1,1,1,1,1)$};
\draw(6,-0.5) node {$(1,1,1,1,1,1)$};
\node at (2.5,-2.5) {Step 1};


\path (11,0) coordinate (a);
\path (12.5,0) coordinate (b);
\path (13.5,1) coordinate (c);
\path (13.5,-1) coordinate (d);
\path (14.5,0) coordinate (e);
\path (16,0) coordinate (f);

\draw (a) -- (b) -- (c) -- (e) -- (f);
\draw (b) -- (d) -- (e);
\draw (c) -- (d);

\draw (a) [fill=white] circle (\vr);
\draw (b) [fill=black] circle (\vr);
\draw (c) [fill=white] circle (\vr);
\draw (d) [fill=white] circle (\vr);
\draw (e) [fill=white] circle (\vr);
\draw (f) [fill=white] circle (\vr);

\draw(10,-0.5) node {$(0,1,1,1,1,1)$};
\draw[anchor = south east] (b) node {$1$};
\draw[anchor = south] (c) node {$(0,1,1,1,1,1)$};
\draw[anchor = north] (d) node {$(0,1,1,1,1,1)$};
\draw(16,0.5) node {$(1,1,1,1,1,1)$};
\draw(17,-0.5) node {$(1,1,1,1,1,1)$};


\path (0,-5) coordinate (a);
\path (1.5,-5) coordinate (b);
\path (2.5,-4) coordinate (c);
\path (2.5,-6) coordinate (d);
\path (3.5,-5) coordinate (e);
\path (5,-5) coordinate (f);

\draw (a) -- (b) -- (c) -- (e) -- (f);
\draw (b) -- (d) -- (e);
\draw (c) -- (d);

\draw (a) [fill=black] circle (\vr);
\draw (b) [fill=white] circle (\vr);
\draw (c) [fill=white] circle (\vr);
\draw (d) [fill=white] circle (\vr);
\draw (e) [fill=white] circle (\vr);
\draw (f) [fill=white] circle (\vr);

\draw[anchor = east] (a) node {$2$};
\draw[anchor = south east] (b) node {$1$};
\draw[anchor = south] (c) node {$(0,0,1,1,1,1)$};
\draw[anchor = north] (d) node {$(0,0,1,1,1,1)$};
\draw(5,-4.5) node {$(1,1,1,1,1,1)$};
\draw(6,-5.5) node {$(1,1,1,1,1,1)$};


\path (11,-5) coordinate (a);
\path (12.5,-5) coordinate (b);
\path (13.5,-4) coordinate (c);
\path (13.5,-6) coordinate (d);
\path (14.5,-5) coordinate (e);
\path (16,-5) coordinate (f);

\draw (a) -- (b) -- (c) -- (e) -- (f);
\draw (b) -- (d) -- (e);
\draw (c) -- (d);

\draw (a) [fill=white] circle (\vr);
\draw (b) [fill=white] circle (\vr);
\draw (c) [fill=black] circle (\vr);
\draw (d) [fill=white] circle (\vr);
\draw (e) [fill=white] circle (\vr);
\draw (f) [fill=white] circle (\vr);

\draw[anchor = east] (a) node {$2$};
\draw[anchor = south east] (b) node {$1$};
\draw[anchor = south] (c) node {$3$};
\draw[anchor = north] (d) node {$(0,0,0,1,1,1)$};
\draw(16,-4.5) node {$(1,1,0,1,1,1)$};
\draw(17,-5.5) node {$(1,1,0,1,1,1)$};
\node at (13.5,-2.5) {Step 2};
\node at (13.5,-7.5) {Step 4};
\node at (13.5,-12) {Step 6};


\path (0,-10) coordinate (a);
\path (1.5,-10) coordinate (b);
\path (2.5,-9) coordinate (c);
\path (2.5,-11) coordinate (d);
\path (3.5,-10) coordinate (e);
\path (5,-10) coordinate (f);

\draw (a) -- (b) -- (c) -- (e) -- (f);
\draw (b) -- (d) -- (e);
\draw (c) -- (d);

\draw (a) [fill=white] circle (\vr);
\draw (b) [fill=white] circle (\vr);
\draw (c) [fill=white] circle (\vr);
\draw (d) [fill=black] circle (\vr);
\draw (e) [fill=white] circle (\vr);
\draw (f) [fill=white] circle (\vr);

\draw[anchor = east] (a) node {$2$};
\draw[anchor = south east] (b) node {$1$};
\draw[anchor = south] (c) node {$3$};
\draw[anchor = north] (d) node {$4$};
\draw(5,-9.5) node {$(1,1,0,0,1,1)$};
\draw(6,-10.5) node {$(1,1,0,0,1,1)$};
\node at (2.5,-7.5) {Step 3};
\node at (2.5,-12) {Step 5};


\path (11,-10) coordinate (a);
\path (12.5,-10) coordinate (b);
\path (13.5,-9) coordinate (c);
\path (13.5,-11) coordinate (d);
\path (14.5,-10) coordinate (e);
\path (16,-10) coordinate (f);

\draw (a) -- (b) -- (c) -- (e) -- (f);
\draw (b) -- (d) -- (e);
\draw (c) -- (d);

\draw (a) [fill=white] circle (\vr);
\draw (b) [fill=white] circle (\vr);
\draw (c) [fill=white] circle (\vr);
\draw (d) [fill=white] circle (\vr);
\draw (f) [fill=white] circle (\vr);
\draw (e) [fill=black] circle (\vr);

\draw[anchor = east] (a) node {$2$};
\draw[anchor = south east] (b) node {$1$};
\draw[anchor = south] (c) node {$3$};
\draw[anchor = north] (d) node {$4$};
\draw[anchor = south west] (e) node {$1$};
\draw(17,-10.5) node {$(0,1,0,0,1,1)$};


\path (6,-14) coordinate (a);
\path (7.5,-14) coordinate (b);
\path (8.5,-13) coordinate (c);
\path (8.5,-15) coordinate (d);
\path (9.5,-14) coordinate (e);
\path (11,-14) coordinate (f);

\draw (a) -- (b) -- (c) -- (e) -- (f);
\draw (b) -- (d) -- (e);
\draw (c) -- (d);

\draw (a) [fill=white] circle (\vr);
\draw (b) [fill=white] circle (\vr);
\draw (c) [fill=white] circle (\vr);
\draw (d) [fill=white] circle (\vr);
\draw (f) [fill=black] circle (\vr);
\draw (e) [fill=white] circle (\vr);

\draw[anchor = east] (a) node {$2$};
\draw[anchor = south east] (b) node {$1$};
\draw[anchor = south] (c) node {$3$};
\draw[anchor = north] (d) node {$4$};
\draw[anchor = south west] (e) node {$1$};
\draw[anchor = west] (f) node {$2$};
\node at (8.5,-16) {Step 7};

\end{tikzpicture}
\end{center}
\caption{Step-by-step execution of Algorithm 1. The vertex shown in black is colored at each step.}
\label{exalg}
\end{figure}

\begin{theorem} \label{compl}
Algorithm 1 computes a packing coloring of a given graph $G$ in ${\cal O}(mn^2)$ time in the worst case, where $n=|V(G)|$ and $m=|E(G)|$. 
\end{theorem}

\begin{proof}
It is clear that Algorithm 1 produces a feasible packing coloring. As for its time complexity, the while loop iterates 
${\cal O}(n)$ time. In each iteration of the while loop, Step 2 takes ${\cal O}(n)$ time in the worst case because an array of length $n$ is traversed until the first nonzero entry is found. BFS algorithm in Step 3 has time complexity ${\cal O}(m)$ in worst case, making the overall worst case time complexity of Algorithm 1 as ${\cal O}(mn^2)$.
\end{proof}



\section{Grundy Packing Chromatic Number} \label{sec:GrundyPackingChromatic}
Since Algorithm 1 produces a feasible packing coloring, it yields an upper bound for $\chi_{\rho}(G)$ and hence it is called a \emph{greedy packing coloring} of $G$. \textcolor{black}{Moreover, every packing coloring can be obtained using Algorithm 1 by selecting a sequence of vertices, starting with those assigned color one, followed by those assigned color two, and so on. Consequently,} $\chi_{\rho}(G)$ is the minimum number of colors that can be obtained by Algorithm 1. We are often interested in finding the other extreme: the maximum numbers of colors in a coloring produced by Algorithm 1, which we refer to as the \emph{Grundy packing chromatic number} of $G$ denoted by $\Gamma_{\rho}(G)$. This provides insight into the quality of the heuristic algorithm and how much reliable it is. An alternative description of $\Gamma_{\rho}(G)$ is simply the maximum number of colors in a packing coloring, where each vertex of color $i\geq 2$ has a vertex of color $j$ at a distance at most $j$ for every $j\in\{1,\dots,i-1\}$.  

We start by applying a polynomial-time transformation to a graph $G$ to obtain a graph $G(k)$, as introduced by Argiroffo et al. \cite{ArNT}. While their original work focused on utilizing this transformation for $\chi_{\rho}(G)$, we demonstrate its applicability to $\Gamma_{\rho}(G)$ with minor adjustments. To ensure that the paper is self-contained, we provide a detailed description of the transformation.

For a graph $G$ we denote by $G^{\ell}$ a graph with $V(G^{\ell})=\{v^{\ell}:v\in V(G)\}$ and $E(G^{\ell})=\{u^{\ell}v^{\ell}:d_G(u,v)\leq \ell\}$. Clearly, $G^1\cong G$ and if $k\geq {\rm diam}(G)$, then $G^k\cong K_n$ for $n=|V(G)|$. Further, for a positive integer $k$ we define graph $G(k)$ by $V(G(k))=\cup_{i=1}^k V(G^i)$ and $E(G(k))=\{v^jv^i:1\leq i<j\leq k\}\cup\left(\cup_{i=1}^k E(G^i)\right)$. Notice that the vertex set \textcolor{black}{$\{v^1,\dots,v^k\}$} constitutes a clique $Q_v$ in $G(k)$ and that every independent set of $G$ contains at most one vertex from this set. \textcolor{black}{Observe graph $P_5(4)$ on Figure \ref{P_5(4)} together with a $i(P_5(4))$-set (black vertices). Notice also that if we color the vertices of $P_5$ corresponding to the black vertices of $P_5^i$ with color $i$, $i\in\{1,2,3,4\}$, then we obtain a packing $4$-coloring of $P_5$. Another example, shown in} Figure \ref{G(3)}, demonstrates graph $G$, its corresponding $G(3)$ together with their minimum maximal independent sets $i(G)$ and $i(G(3))$, denoted by black vertices.

\begin{figure}[htb]
\begin{center}
\begin{tikzpicture}[xscale=.8, yscale=.8, style=thick,x=1cm,y=1cm]
\def\vr{3.5pt} 

\path (-2,2) coordinate (a1);
\path (-1,1) coordinate (a2);
\path (0,0) coordinate (a3);
\path (1,-1) coordinate (a4);
\path (2,-2) coordinate (a5);

\path (5,-2) coordinate (b5);
\path (6,-1) coordinate (b4);
\path (7,0) coordinate (b3);
\path (8,1) coordinate (b2);
\path (9,2) coordinate (b1);

\path (9,5) coordinate (c1);
\path (8,6) coordinate (c2);
\path (7,7) coordinate (c3);
\path (6,8) coordinate (c4);
\path (5,9) coordinate (c5);

\path (-2,5) coordinate (d1);
\path (-1,6) coordinate (d2);
\path (0,7) coordinate (d3);
\path (1,8) coordinate (d4);
\path (2,9) coordinate (d5);
\draw (a1) -- (a5) -- (b5) -- (c5) -- (d5) -- (a5) -- (c5);
\draw (d1) -- (b1) -- (c1) -- (d1) -- (a1) -- (b1) -- (b5) -- (d5);
\draw (a1) -- (c1) -- (c5);
\draw (d1) -- (d5);
\draw (a2) -- (b2) -- (c2) -- (d2) -- (a2) -- (c2);
\draw (d3) -- (b3) -- (c3) -- (d3) -- (a3) -- (b3);
\draw (d2) -- (b2);
\draw (a3) -- (c3);
\draw (d4) -- (b4) -- (c4) -- (d4) -- (a4) -- (b4);
\draw (c4) -- (a4);

\draw (b5) to [bend right] (b3);
\draw (b4) to [bend right] (b2);
\draw (b3) to [bend right] (b1);
\draw (c5) to [bend left] (c3);
\draw (c4) to [bend left] (c2);
\draw (c3) to [bend left] (c1);
\draw (c4) to [bend left] (c1);
\draw (c5) to [bend left] (c2);
\draw (d1) to [bend left] (d3);
\draw (d1) to [bend left] (d4);
\draw (d1) to [bend left] (d5);
\draw (d2) to [bend left] (d4);
\draw (d2) to [bend left] (d5);
\draw (d3) to [bend left] (d5);

\draw (a1) [fill=white] circle (\vr); \draw (a2) [fill=black] circle (\vr);
\draw (a3) [fill=white] circle (\vr); \draw (a4) [fill=white] circle (\vr);
\draw (a5) [fill=black] circle (\vr); \draw (b1) [fill=white] circle (\vr);
\draw (b2) [fill=white] circle (\vr); \draw (b3) [fill=black] circle (\vr);
\draw (b4) [fill=white] circle (\vr); \draw (b5) [fill=white] circle (\vr);
\draw (c1) [fill=black] circle (\vr); \draw (c2) [fill=white] circle (\vr);
\draw (c3) [fill=white] circle (\vr); \draw (c4) [fill=white] circle (\vr);
\draw (c5) [fill=white] circle (\vr); \draw (d1) [fill=white] circle (\vr);
\draw (d2) [fill=white] circle (\vr); \draw (d3) [fill=white] circle (\vr);
\draw (d4) [fill=black] circle (\vr); \draw (d5) [fill=white] circle (\vr);

\draw(3.5,-3) node {$P_5(4)$};
\draw(-0.5,-.5) node {$P_5^1$};
\draw(8,-.5) node {$P_5^2$};
\draw(8,8) node {$P_5^3$};
\draw(-0.7,8) node {$P_5^4$};

\end{tikzpicture}
\end{center}
\caption{Graph $P_5(4)$ together with $i(P_5(4))$-set (black vertices).}
\label{P_5(4)}
\end{figure}

\begin{figure}[htb]
\begin{center}
\begin{tikzpicture}[xscale=.8, yscale=.8, style=thick,x=1cm,y=1cm]
\def\vr{3.5pt} 


\path (-1,1) coordinate (a);
\path (0.5,1) coordinate (b);
\path (2,1) coordinate (c);
\path (3.5,1) coordinate (d);
\path (-0.25,2) coordinate (e);
\path (0.5,3) coordinate (f);
\path (2,3) coordinate (g);
\path (2.75,2) coordinate (h);
\path (1.25,4) coordinate (i);
\path (1.25,2) coordinate (j);

\draw (a) -- (b) -- (c) -- (d) -- (h) -- (g) -- (i) -- (f) -- (e) -- (a);
\draw (b) -- (h) -- (f) -- (b);
\draw (c) -- (j) -- (e);
\draw (j) -- (g);

\draw (a) [fill=white] circle (\vr);
\draw (b) [fill=white] circle (\vr);
\draw (c) [fill=black] circle (\vr);
\draw (d) [fill=white] circle (\vr);
\draw (e) [fill=black] circle (\vr);
\draw (f) [fill=white] circle (\vr);
\draw (g) [fill=black] circle (\vr);
\draw (h) [fill=white] circle (\vr);
\draw (i) [fill=white] circle (\vr);
\draw (j) [fill=white] circle (\vr);

\draw(1.25,0.2) node {$G$};


\path (8,2) coordinate (a);
\path (9.5,2) coordinate (b);
\path (11,2) coordinate (c);
\path (12.5,2) coordinate (d);
\path (8.75,3) coordinate (e);
\path (9.5,4) coordinate (f);
\path (11,4) coordinate (g);
\path (11.75,3) coordinate (h);
\path (10.25,5) coordinate (i);
\path (10.25,3) coordinate (j);

\draw (a) -- (b) -- (c) -- (d) -- (h) --(g) -- (i) -- (f) -- (e) -- (a);
\draw (b) -- (h) -- (f) -- (b);
\draw (c) -- (j) -- (e);
\draw (j) -- (g);


\path (4,-2) coordinate (a1);
\path (5.5,-2) coordinate (b1);
\path (7,-2) coordinate (c1);
\path (8.5,-2) coordinate (d1);
\path (4.75,-1) coordinate (e1);
\path (5.5,0) coordinate (f1);
\path (7,0) coordinate (g1);
\path (7.75,-1) coordinate (h1);
\path (6.25,1) coordinate (i1);
\path (6.25,-1) coordinate (j1);

\draw (a1) -- (b1) -- (c1) -- (d1) -- (h1) --(g1) -- (i1) -- (f1) -- (e1) -- (a1);
\draw (b1) -- (h1) -- (f1) -- (b1);
\draw (c1) -- (j1) -- (e1);
\draw (j1) -- (g1);
\draw (g1) to [bend left] (d1);
\draw (i1) to [bend left] (h1);
\draw (a1) to [bend left] (f1);
\draw (e1) to [bend left] (i1);
\draw (a1) to [bend right] (c1);
\draw (b1) to [bend right] (d1);
\draw (e1) to [bend right] (h1);
\draw (c1) to [bend left] (f1);
\draw (g1) to [bend left] (b1);
\draw (e1) -- (d1) -- (f1) -- (g1) -- (c1) -- (i1) -- (b1) -- (e1) -- (g1) -- (a1) -- (h1);
\draw (b1) -- (j1) -- (f1);
\draw (e1) -- (c1) -- (h1) -- (j1) -- (a1);
\draw (d1) -- (j1) -- (i1);

\path (13,-3) coordinate (a2);
\path (14.5,-3) coordinate (b2);
\path (16,-3) coordinate (c2);
\path (17.5,-3) coordinate (d2);
\path (13.75,-2) coordinate (e2);
\path (14.5,-1) coordinate (f2);
\path (16,-1) coordinate (g2);
\path (16.75,-2) coordinate (h2);
\path (15.25,0) coordinate (i2);
\path (15.25,-2) coordinate (j2);

\draw (a2) -- (b2) -- (c2) -- (d2) -- (h2) --(g2) -- (i2) -- (f2) -- (e2) -- (a2);
\draw (b2) -- (h2) -- (f2) -- (b2);
\draw (c2) -- (j2) -- (e2);
\draw (j2) -- (g2);
\draw (g2) to [bend left] (d2);
\draw (i2) to [bend left] (h2);
\draw (a2) to [bend left] (f2);
\draw (e2) to [bend left] (i2);
\draw (a2) to [bend right] (c2);
\draw (b2) to [bend right] (d2);
\draw (e2) to [bend right] (h2);
\draw (c2) to [bend left] (f2);
\draw (g2) to [bend left] (b2);
\draw (e2) -- (d2) -- (f2) -- (g2) -- (c2) -- (i2) -- (b2) -- (e2) -- (g2) -- (a2) -- (h2);
\draw (b2) -- (j2) -- (f2);
\draw (e2) -- (c2) -- (h2) -- (j2) -- (a2);
\draw (d2) -- (j2) -- (i2);

\draw (i2) to [bend left] (d2);
\draw (a2) to [bend left] (i2);
\draw (a2) to [bend right] (d2);

\draw (a2) -- (a1) -- (a) -- (a2);
\draw (b2) -- (b1) -- (b) -- (b2);
\draw (c2) -- (c1) -- (c) -- (c2);
\draw (d2) -- (d1) -- (d) -- (d2);
\draw (e2) -- (e1) -- (e) -- (e2);
\draw (f2) -- (f1) -- (f) -- (f2);
\draw (g2) -- (g1) -- (g) -- (g2);
\draw (h2) -- (h1) -- (h) -- (h2);
\draw (i2) -- (i1) -- (i) -- (i2);
\draw (j2) -- (j1) -- (j) -- (j2);

\draw (a) [fill=white] circle (\vr);
\draw (b) [fill=white] circle (\vr);
\draw (c) [fill=black] circle (\vr);
\draw (d) [fill=white] circle (\vr);
\draw (e) [fill=black] circle (\vr);
\draw (f) [fill=white] circle (\vr);
\draw (g) [fill=black] circle (\vr);
\draw (h) [fill=white] circle (\vr);
\draw (i) [fill=white] circle (\vr);
\draw (j) [fill=white] circle (\vr);

\draw (a1) [fill=white] circle (\vr);
\draw (b1) [fill=black] circle (\vr);
\draw (c1) [fill=white] circle (\vr);
\draw (d1) [fill=white] circle (\vr);
\draw (e1) [fill=white] circle (\vr);
\draw (f1) [fill=white] circle (\vr);
\draw (g1) [fill=white] circle (\vr);
\draw (h1) [fill=white] circle (\vr);
\draw (i1) [fill=white] circle (\vr);
\draw (j1) [fill=white] circle (\vr);

\draw (a2) [fill=white] circle (\vr);
\draw (b2) [fill=white] circle (\vr);
\draw (c2) [fill=white] circle (\vr);
\draw (d2) [fill=white] circle (\vr);
\draw (e2) [fill=white] circle (\vr);
\draw (f2) [fill=white] circle (\vr);
\draw (g2) [fill=white] circle (\vr);
\draw (h2) [fill=black] circle (\vr);
\draw (i2) [fill=white] circle (\vr);
\draw (j2) [fill=white] circle (\vr);
\draw(10,-4.2) node {$G(3)$};
\draw(12,4) node {$G^1$};
\draw(6,-3) node {$G^2$};
\draw(16,0) node {$G^3$};

\end{tikzpicture}
\end{center}
\caption{An example graph $G$ and its corresponding $G(3)$ together with their $i(G)$- and $i(G(3))$-sets, respectively (black vertices).}
\label{G(3)}
\end{figure}

The following crucial lemma, essential for the existence of a packing $k$-coloring, was presented in \cite{ArNT}. 

\begin{lemma} \label{paccol}
\cite{ArNT} Let $G$ be a graph on $n$ vertices and $k\in\{1,\dots,n\}$. A graph $G$ admits a packing $k$-coloring if and only if there exists an independent set of cardinality $n$ in $G(k)$.
\end{lemma}

Lemma \ref{paccol} can be applied to derive $\chi_{\rho}(G)$ by minimizing $k$, as demonstrated in Lemma 3 of \cite{ArNT}, which states the following:

\begin{equation}\label{others}
\chi_{\rho}(G)=\min\{k:\alpha(G(k))=n\}.
\end{equation}

Note here that the approach of maximizing $k$ and replacing $\alpha(G)$ by $i(G)$ to obtain $\Gamma_{\rho}(G)$ is not viable, as the maximum value of $k$ for which there exists an independent set of cardinality $n$ is always $n$. To accomplish this, one can merely place a distinct vertex from $G$ into each copy of $G^i$ for every $i\in\{1,\dots,n\}$, resulting in an independent set of $G(n)$. Indeed, this forms a maximal independent set as each $Q_v$, where $v\in V(G)$, constitutes a clique within $V(G(n))$. Such an independent set yields a packing coloring with $n$ colors; however, typically, this does not represent a Grundy packing coloring. So, the question arises: which packing colorings also qualify as greedy packing colorings? We address this query by adapting Lemma \ref{paccol} as follows.

\begin{lemma} \label{pacgreedycol}
Let $G$ be a graph on $n$ vertices and $k\in\{1,\dots,n\}$. A graph $G$ admits a greedy packing $k$-coloring if and only if there exists an independent set $A$ of cardinality $n$ in $G(k)$ such that $A^i=A\cap G^i$ is a maximal independent set of $G^i-\{v^i:v^j\in A^j, j\in\{1,\dots,i-1\}\}$ for every $i\in\{1, \dots, k-1\}$.
\end{lemma}

\begin{proof}
Suppose that there exists a greedy packing $k$-coloring $c$ of a graph $G$ with color classes $V_1,\dots,V_k$. Let $A^i$ be a subset of $V(G^i)$ that corresponds to vertices of $V_i$ for every $i\in\{1,\dots,k\}$ and let $A=A^1\cup\cdots\cup A^k$. Suppose that $A^i$ is not a maximal independent set of $G^i-\{v^i:v^j\in A^j, j\in\{1,\dots,i-1\}\}$ for some $i\in\{1,\dots,k-1\}$. Hence, there exists a vertex $z^{i}\in V(G^i)-\{v^i:v^j\in A^j, j\in\{1,\dots,i-1\}\}$ such that $A^i\cup\{z^{i}\}$ is an independent set of $G^i-\{v^i:v^j\in A^j, j\in\{1,\dots,i-1\}\}$. But then there exists a vertex $z\notin V_1\cup \cdots\cup V_i$ such that $c(z)>i$. This yields a contradiction, since $z$ has no vertex of color $i$ within a distance of at most $i$ at the moment it receives the color in Algorithm 1.

Conversely, let $A$ be an independent set of cardinality $n$ in $G(k)$ such that $A^i=A\cap G^i$ is a maximal independent set of $G^i-\{v^i:v^j\in A^j, j\in\{1,\dots,i-1\}\}$ for every $i\in\{1,\dots,k-1\}$. Each set $V_i\subset V(G)$ contains the corresponding vertices of $A^i$. We claim that $V_1,\dots,V_k$ yield a greedy packing $k$-coloring $c$ of $G$. If we order $V(G)$ according to their colors with respect to $c$, then this gives a possible order in Algorithm 1. When $v\in V_i$, for some $i\in\{1,\dots,k\}$, is considered by this order in Algorithm 1, it has every color $j\in\{1,\dots,i-1\}$ at distance at most $j$ because $A^j$ is a maximal independent set of $G^j-\{v^j:v^{\ell}\in A^{\ell}, \ell\in\{1,\dots,j-1\}\}$ for every $\ell\in\{1,\dots,j-1\}$. Hence, $v$ receives color $i$ by Algorithm 1, and this yields a greedy packing $k$-coloring of $G$.  
\end{proof}

Recall that the vertex set $v^1,\dots,v^k$ forms a clique $Q_v$ in $G(k)$. Notice that every independent set $A$ of $G(k)$ with cardinality $n=|V(G)|$ is a maximal independent set of $G(k)$ because it can have at most one vertex from $Q_v$ for every $v\in V(G)$. However, the condition of Lemma \ref{pacgreedycol} is not always fulfilled. Therefore, we introduce a procedure called \emph{dense maximization procedure} (DMP for short) for an independent set $A$ of $G(k)$ with cardinality $n$. If $A^i=A\cap G^i$ is a maximal independent set of $G^i-\{v^i:v^j\in A^j, j\in\{1,\dots,i-1\}\}$ for every $i\in\{1,\dots,k-1\}$, then we are done. Otherwise, consider the minimum $i\in\{1,\dots,k-1\}$ such that $A^i$ is not a maximal independent set of $G^i-\{v^i:v^j\in A^j, j\in\{1,\dots,i-1\}\}$. Hence, there exists $z^{\ell}\in A$ for some $\ell>i$ such that $A_i\cup \{z^i\}$ is independent in $G^i-\{v^i:v^j\in A^j, j\in\{1,\dots,i-1\}\}$. We exchange $z^{\ell}$ with $z^i$ in $A$ and keep the notation $A$. Clearly, $A$ is independent and $|A|=n$ after this change. We continue this procedure until $A^i$ is a maximal independent set of $G^i-\{v^i:v^j\in A^j, j\in\{1,\dots,i-1\}\}$. Next, we proceed with the first $t>i$, where $A^t$ is not a maximal independent set of $G^i-\{v^i:v^j\in A^j, j\in\{1,\dots,t-1\}\}$, if such a set exists. At the end of the DMP, $A\cap V(G^k)$ may be empty, and whether it is empty or not plays an important role. Therefore, we denote by $DMP(A)$ the maximum $\ell$ such that $A\cap V(G^{\ell})\neq \emptyset$ after any run of the DMP.

We are now able to present an exact description of $\Gamma_{\rho}(G)$ that will lead to several significant consequences. This result is somewhat connected to Equation (\ref{others}), as we are seeking a maximum among some special independent sets of $G(k)$.

\begin{theorem} \label{exact}
Let $G$ be a graph with $n=|V(G)|$. If $\cal{I}$ is the set of all independent sets of $G(k)$ with cardinality $n$ for any possible integer $k\leq n$, then 
$$\Gamma_{\rho}(G)=\max_{A\in\cal{I}}\{DMP(A)\}.$$
\end{theorem}

\begin{proof}
Let $t=\max_{A\in\cal{I}}\{DMP(A)\}$. Clearly, there exist independent sets of cardinality $n$ in $G(n)$ and let $A\in\cal{I}$ be such that $DMP(A)=t$. Suppose that $A'$ is the set obtained by DMP from $A$ such that $A'\cap G^t\neq \emptyset$. The set $A'$ fulfills the condition of Lemma \ref{pacgreedycol}, and hence there exists a Grundy packing $t$-coloring by the same lemma. Therefore, $\Gamma_{\rho}(G)\geq t$. 

Conversely, $\Gamma_{\rho}(G)$ cannot exceed $t$ because a greedy packing $\Gamma_{\rho}(G)$-coloring would yield an independent set $A^*$ of $G(s)$ of cardinality $n$ that fulfills the condition of Lemma \ref{pacgreedycol} and has $A^{*}\cap G(s)\neq\emptyset$ for some $s>t$. This is a contradiction because $A^{*}\in\cal{I}$ and $s>t$. Therefore, the equality follows.   
\end{proof}


An independent set $A$ of cardinality $|V(G)|$ of $G(\Gamma_{\rho}(G))$ that yields $\Gamma_{\rho}(G)$ by Theorem \ref{exact} will simply be denoted by a $G(\Gamma_{\rho}(G))$-set. A $G(\Gamma_{\rho}(G))$-set $A$ contains at least $i(G)$ elements in $G^1$ by the condition of Lemma \ref{pacgreedycol}. Additionally, $|A\cap G^j|\geq 1$ for every $j\in \{2,\dots,\Gamma_{\rho}(G)\}$. Hence, the following upper bound is valid for $\Gamma_{\rho}(G)$, which is symmetrical to the one for $\chi_{\rho}(G)$ from \cite{GHHHR}, with the exchange of $\alpha(G)$ by $i(G)$.

\begin{corollary} \label{upbound}
If $G$ is a graph, then $\Gamma_{\rho}(G)\leq n-i(G)+1$.
\end{corollary}

\textcolor{black}{Some graphs for which the equality is satisfied will be presented in Section 5.}

We end this section with a short discussion about the complexity of $\Gamma_{\rho}(G)$. We consider the following decision problem.

\begin{equation*}
\begin{tabular}{|l|}
\hline
\mbox{GRUNDY PACKING COLORING PROBLEM (GPC problem for short)} \\
\mbox{INSTANCE: A graph $G$ on $n$ vertices and an integer $1\le k\leq n$.} \\
\mbox{QUESTION: Is $\Gamma_{\rho}(G)\leq k$?} \\ \hline
\end{tabular}%
\end{equation*}%

Since it is NP-complete to decide whether $i(G)\geq \ell$ even when $G$ is restricted to bipartite graphs, line graphs, circle graphs, unit disk graphs or planar cubic graphs (see Section 7 of \cite{GoHe} and the references therein), the following theorem holds by Corollary \ref{upbound}. 

\begin{theorem}\label{NPcomp}
GPC problem is NP-complete even when $G$  is restricted to bipartite graphs, line graphs, circle graphs, unit disk graphs or planar cubic graphs.
\end{theorem}


\section{Graphs with Large Grundy Packing Chromatic Number} \label{sec:LargePackingGrundy}

In this section we describe all graphs with $\Gamma_{\rho}(G)\in \{n-1,n\}$, where $n$ denotes the number of vertices of the graph under consideration. 
We start with all graphs for which Algorithm 1 can produce a coloring where every vertex receives a unique color and the class is much larger than for the packing chromatic number, where only complete graphs fulfill this property.

\begin{proposition} \label{ncol}
A graph $G$ on $n$ vertices has $\Gamma_{\rho}(G)=n$ if and only if $G$ has a universal vertex.
\end{proposition}

\begin{proof}
Let first $\Gamma_{\rho}(G)=n$. So, every color is used exactly once and let $v\in V(G)$ be of color $1$. If there exists $u\in V(G)$ that is not adjacent to $v$, then Step 2 of Algorithm 1 would give color $1$ also to $v$, a contradiction. Thus, $v$ is a universal vertex.

If $G$ has a universal vertex, then $i(G)=1$ and $G^i\cong K_n$ for every $i>1$. So, there exists a $G(\Gamma_{\rho}(G))$-set $A$ with $A\cap G^n\neq\emptyset$. Hence, by Theorem \ref{exact} we have $\Gamma_{\rho}(G)=n$.
\end{proof}

Next, we concentrate on graphs where $\Gamma_{\rho}(G)=n-1$. We distinguish between the connected and disconnnected cases. Observe that $\Gamma_{\rho}(G)$ of a disconnected graph is equal to the maximum value $\Gamma_{\rho}(G_i)$ of its components $G_1,\dots,G_k$. For disconnected graphs, the following is clear by Proposition \ref{ncol}:

\begin{corollary} \label{}
A disconnected graph $G$ with $n$ vertices has $\Gamma_{\rho}(G)=n-1$ if and only if $G=K_1 \cup G'$, where $G'$ is a graph having a universal vertex.
\end{corollary}

We proceed with the following simple lemma that will assist us in describing all connected graphs with $\Gamma_{\rho}(G)=n-1$.

\begin{lemma} \label{raddiam}
Let $G$ be a connected graph. If $i(G)=2$, then ${\rm rad}(G)\in\{2,3\}$ and ${\rm diam}(G)\in\{2,3,4,5\}$.
\end{lemma}

\begin{proof}
Let $G$ be a connected graph with $i(G)=2$. Hence, $G$ has no universal vertex and ${\rm rad}(G)>1$ and ${\rm diam}(G)>1$. For an $i(G)$-set $\{x,y\}$ we have $2\leq d(x,y)\leq 3$. If $d(x,y)=2$, then $x$ and $y$ have a common neighbor $w$ with $d(w,v)\leq 2$ for every $v\in V(G)$. Thus, ${\rm rad}(G)=2$ and consequently ${\rm diam}(G)\leq 4$. So, let $d(x,y)=3$. Any neighbor $w$ of $x$ on a shortest $x,y$-path has $d(w,v)\leq 3$ for every $v\in V(G)$. Thus ${\rm ecc}(w)\leq 3$ and ${\rm rad}(G)\leq 3$ follows. From this, it also follows that ${\rm diam}(G)\leq 6$. If ${\rm diam}(G)=6$, then there exists a diametrical path $P$ of length six. At most one vertex from $x$ and $y$ can be on $P$. If $x$ is on $P$, then $y$ is adjacent to at least four vertices of $P$, a contradiction with $P$ being a shortest path. Otherwise, if $x$ and $y$ do not belong to $P$, then by the pigeon hole principle, one of them must be again adjacent to at least four of them, again a contradiction. Hence, the result holds.
\end{proof}

\begin{theorem} \label{n-1}
A connected graph $G$ on $n$ vertices has $\Gamma_{\rho}(G)=n-1$ if and only if the following statements hold:
\begin{itemize}
\item[(i)] $i(G)=2$. 
\item[(ii)] ${\rm diam}(G)\leq 4$.
\item[(iii)] If ${\rm rad}(G)=3$, then there exist an $i(G)$-set $\{x,y\}$ with $d(x,y)=3$ and $w\in N(x)$ such that $d(w,z)\leq 2$ for every vertex $z\in N(y)$.
\item[(iv)] If ${\rm rad}(G)=2$ and ${\rm diam}(G)=4$, then there exists an $i(G)$-set that avoids one central vertex and one additional non-diametrical vertex.
\end{itemize}
\end{theorem}

\begin{proof}
Suppose first that $(i),(ii),(iii)$ and $(iv)$ hold. Let $A=\{x,y\}$ be an $i(G)$-set. By Lemma \ref{raddiam} we have ${\rm rad}(G)\in\{2,3\}$. Suppose first that ${\rm rad}(G)=3$. By $(iii)$ we set $c(x)=1=c(y)$ and $c(w)=2$. There also exists a vertex $z\notin \{x,y,w\}$ that belongs to a shortest $x,y$-path. So, $c(z)=3$ and every other vertex receives a different color since ${\rm diam}(G)\leq 4$ by $(ii)$ and all other vertices are at distance at most $2$ to $w$, yielding a packing $(n-1)$-coloring of $G$. Since $\Gamma_{\rho}(G)<n$ by Proposition \ref{ncol}, we have $\Gamma_{\rho}(G)=n-1$.
 
Now let ${\rm rad}(G)=2$. If in addition ${\rm diam}(G)=4$, then we may assume that $A$ fulfills $(iv)$. Accordingly, let $w\notin A$ be the central vertex and $z\notin A$, $z\neq w$ be a non-diametrical vertex. If we set $c(x)=1=c(y)$, $c(w)=2$, $c(z)=3$ and every other vertex receives a different color, then $c$ is a packing $(n-1)$-coloring and $\Gamma_{\rho}(G)\geq n-1$. Equality now follows by Proposition \ref{ncol}. So, let ${\rm diam}(G)=3$. If there exists a central vertex $w\notin A$, then we set $c(x)=1=c(y)$, $c(w)=2$ and every other vertex receives a different color, yielding an $(n-1)$-packing coloring and $\Gamma_{\rho}(G)=n-1$ by Proposition \ref{ncol}. Otherwise, either $x$ or $y$, say $x$, is a central vertex and $d(x,y)=2$. Let $w$ be a common neighbor of $x$ and $y$ and we have $d(w,v)\leq 2$ for every $v\in V(G)$ since $A$ is an $i(G)$-set. So, a coloring where $c(x)=1=c(y)$, $c(w)=2$ and every other vertex receives a different color is an $(n-1)$-packing coloring and again $\Gamma_{\rho}(G)=n-1$ by the same reason. If ${\rm diam}(G)=2$, then we can use the same coloring where $w$ is a common neighbor of $x$ and $y$ and we are done with this direction.

Conversely, suppose first that $(i)$ does not hold. If $i(G)=1$, then $\Gamma_{\rho}(G)=n$ by Proposition \ref{ncol}. If $i(G)\geq 3$, then $\Gamma_{\rho}(G)\leq n-i(G)+1\leq n-2$ by Corollary \ref{upbound}, a contradiction. Therefore, $i(G)=2$. Now let $A=\{x,y\}$ be an $i(G)$-set and $B$ be a $G(\Gamma_{\rho}(G))$-set. Recall that $G^1$ and $G^2$ are subgraphs of $G(\Gamma_{\rho}(G))$. If $|B\cap V(G^1)|>i(G)$, then we have $\Gamma_{\rho}(G)\leq n-|B\cap V(G^1)|+1<n-i(G)+1=n-1$ and $\Gamma_{\rho}(G)\neq n-1$. Therefore, we may also assume that $|B\cap V(G^1)|=i(G)$ holds. 

By Lemma \ref{raddiam} we have ${\rm diam}(G)\in\{2,3,4,5\}$. So, if $(ii)$ does not hold, then ${\rm diam}(G)=5$. Let $v_1v_2v_3v_4v_5v_6$ be a diametrical path of $G$. Since $A$ is an $i(G)$-set, we may assume that $x$ is either $v_2$ or is adjacent to $v_1,v_2$ and $v_3$, and $y$ is either $v_5$ or it is adjacent to $v_4,v_5$ and $v_6$. In particular, $x$ and $y$ are different from $v_1$ and $v_6$ for any $i(G)$-set. Notice that no vertex in $G$ is at distance two to $v_1$ and $v_6$ at the same time and hence $|B\cap V(G^2)|>1$ follows. Therefore, $\Gamma_{\rho}(G)\leq n-i(G)-|B\cap V(G^2)|+2\leq n-2$. Thus, $\Gamma_{\rho}(G)\neq n-1$. 

If $(iii)$ is false, then ${\rm rad}(G)=3$ and for every $i(G)$-set $\{x,y\}$ with $d(x,y)=3$ and every $w\in N(x)$ there exists $z\in N(y)$ with $d(w,z)>2$. This means that no matter how we choose $w$ there always exists $z$ such that they are not adjacent in $G^2$. Again we have $|B\cap V(G^2)|>1$ and $\Gamma_{\rho}(G)\leq n-i(G)-|B\cap V(G^2)|+2\leq n-2$ which ends this case.

Finally, suppose that $(iv)$ does not hold. Then ${\rm rad}(G)=2$ and ${\rm diam}(G)=4$ and every $i(G)$-set either contains all central vertices or it does not contain one central vertex but contains all the other non-diametrical vertices. Alternative description is that outside of any $i(G)$-set, there exists at most one non-diametrical vertex $w$. If we have $c(x)=c(y)=1$ and $c(w)=2$, then any diametrical pair is nonadjacent in $G^3$ and $|B\cap V(G^3)|>1$. Hence, $\Gamma_{\rho}(G)\leq n-i(G)-1-|B\cap V(G^3)|+3\leq n-2$. Otherwise, if $c(x)=1=c(y)$, $c(w)\neq 2$, then $|B\cap V(G^2)|>1$ and $\Gamma_{\rho}(G)\leq n-i(G)-|B\cap V(G^2)|+2\leq n-2$ follows. Consequently, $\Gamma_{\rho}(G)\neq n-1$ in both cases.
\end{proof}

Graphs that are eliminated by $(iv)$ of Theorem \ref{n-1} are relatively rare. One can obtain them from two stars $K_{1,r}$ and $K_{1,t}$, $r,t>1$, where we identify one leaf $u$ of the first with one leaf $u'$ of the second star. In addition we can add arbitrary edges between the leaves of the first star (excluding $u$) and arbitrary edges between leaves of the second star ($u'$ excluded). Besides, a graph that fulfills $(iii)$ of Theorem \ref{n-1} is on Figure \ref{example(iii)}.

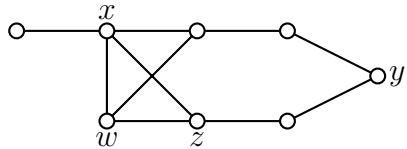
\begin{figure}[H]
\begin{center}
\begin{tikzpicture}[xscale=.8, yscale=.8, style=thick,x=1cm,y=1cm]
\def\vr{3.5pt} 


\path (0,0) coordinate (a);
\path (0,1.5) coordinate (b);
\path (1.5,0) coordinate (c);
\path (1.5,1.5) coordinate (d);
\path (3,1.5) coordinate (e);
\path (3,0) coordinate (f);
\path (4.5,0.75) coordinate (g);
\path (-1.5,1.5) coordinate (h);


\draw (h) -- (b) -- (a) -- (c) -- (b) -- (d) -- (a);
\draw (d) -- (e) -- (g) -- (f) -- (c);

\draw (a) [fill=white] circle (\vr);
\draw (b) [fill=white] circle (\vr);
\draw (c) [fill=white] circle (\vr);
\draw (d) [fill=white] circle (\vr);
\draw (e) [fill=white] circle (\vr);
\draw (f) [fill=white] circle (\vr);
\draw (g) [fill=white] circle (\vr);
\draw (h) [fill=white] circle (\vr);

\draw[anchor = north] (a) node {$w$};
\draw[anchor = south] (b) node {$x$};
\draw[anchor = west] (g) node {$y$};
\draw[anchor = north] (c) node {$z$};

\end{tikzpicture}
\end{center}
\caption{Graph $G$ that fulfills $(iii)$ of Theorem \ref{n-1}. }
\label{example(iii)}
\end{figure}


\section{Graphs with Small Diameter}\label{sec:SmallDiam}

This section is devoted to $\Gamma_{\rho}(G)$ of graphs with ${\rm diam}(G)\in\{1,2,3\}$. More precisely, we can ignore complete graphs, which are the only graphs of diameter one because obviously $\Gamma_{\rho}(K_n)=n$. We proceed with the assumption that \({\rm diam}(G) = 2\). The next result can be seen as an analogue of Equation (\ref{diam2}), a generalization of Proposition \ref{ncol} or a special case of Corollary \ref{upbound}. This follows from Theorem \ref{exact}, given that \(G^j \cong K_n\) for every \(j \geq 2\).

\begin{corollary} \label{diam22}
If $G$ is a graph with ${\rm diam}(G)=2$, then $\Gamma_{\rho}(G)=n-i(G)+1$.
\end{corollary}

Well-known diameter two graphs are joins $G\vee H$ of graphs $G$ and $H$ (unless both $G$ and $H$ are complete graphs), which are graphs with $V(G\vee H)=V(G)\cup V(H)$ and $E(G\vee H)=E(G)\cup E(H)\cup \{gh:g\in V(G),h\in V(H)\}$. It is easy to see that $i(G\vee H)=\min\{i(G),i(H)\}$ which yields the following corollary.

\begin{corollary} \label{join}
For graphs $G$ and $H$, $\Gamma_{\rho}(G\vee H)=|V(G)|+|V(H)|-\min\{i(G),i(H)\}+1$. In particular, for $s,t\geq 1$, $p,r\geq 4$ and $n\geq 2$, we have: 
\begin{itemize}
\item $\Gamma_{\rho}(K_{s,t})=s+t-\min\{s,t\}+1$;
\item $\Gamma_{\rho}(G\vee K_1)=|V(G)|+1$;
\item $\Gamma_{\rho}(K_s\vee \overline{K}_n)=s+n$;
\item $\Gamma_{\rho}(P_p\vee P_r)=\Gamma_{\rho}(P_p\vee C_r)=\Gamma_{\rho}(C_p\vee C_r)=p+r-\min\left\{\left\lceil\frac{p}{3} \right\rceil, \left\lceil\frac{r}{3} \right\rceil\right\}+1$. 
\end{itemize}
\end{corollary}

We now introduce a construction that is useful for presenting an alternative description of $\Gamma_{\rho}(G)$, where ${\rm diam}(G)=3$. It is obvious that $G^j\cong K_n$ for $n=|V(G)|$ and every $j\geq 3$. Hence, in view of Theorem \ref{exact}, where $B$ is a $G(\Gamma_{\rho}(G))$-set, only $B\cap V(G^1)$ and $B\cap V(G^2)$ can have more than one vertex. Moreover, if $x,y\in B\cap V(G^2)$, then they form a diametrical pair. Therefore, we define a diametrical graph $D(G)$ of $G$, where $V(D(G))=V(G)$ and two vertices $x$ and $y$ are adjacent in $D(G)$ if they form a diametrical pair of $G$. Notice that the graph $G$ in Figure \ref{G(3)} contains three diametrical vertices (all of degree two), mutually forming diametrical pairs, and $D(G)\cong K_3\cup 7K_1$. Similarly, we can see that $D(C_6)=3K_2$ and $D(Q_3)=4K_2$ where $Q_3$ is a cube graph. With $D(G)$, we can describe $\Gamma_{\rho}(G)$ for any diameter three graph as an optimization problem over all maximal independent sets of $G$.     

\begin{theorem} \label{diam3}
Let $G$ be a graph and let $\cal{M}$ be the set of all maximal independent sets of $G$. If ${\rm diam}(G)=3$, then $\Gamma_{\rho}(G)=|V(G)|-m(G)+2$, where 
$$m(G)=\min_{A\in\cal{M}}\{|A|+|Q|:Q \text{ is a maximal clique of minimum cardinality of }D(G)-A\}.$$
\end{theorem}

\begin{proof}
Let $G$ be a graph with $|V(G)|=n$ and ${\rm diam}(G)=3$. Let $A_0\in \cal{M}$ such that $m(G)=|A_0|+|Q|$, where $Q$ is a maximal clique of minimum cardinality of $D(G)-A_0$. Notice that not all vertices of $G$ belong to $A_0$ since ${\rm diam}(G)=3$. Therefore, $D(G)-A_0$ contains some vertices and $m(G)$ is well defined. We construct $B\subseteq V(G(k))$ for $1 \le k \le n$ as follows. All the vertices from $G^1$ that correspond to vertices of $A_0$ in $G$ are put in $B$. Similarly, we put the corresponding vertices from $Q$ in $G^2$ to $B$. In addition, $B$ gets exactly one corresponding vertex from $V(G)-(A_0\cup Q)$ for every other $G^j$, $j\geq 3$, until all vertices are used. Clearly, $B$ is an independent set of $G(k)$ and $k=|V(G)|-|A_0|-|Q|+2$ when $|B|=n$. By Theorem \ref{exact}, we get $\Gamma_{\rho}(G)\geq |V(G)|-|A_0|-|Q|+2=|V(G)|-m(G)+2$. 

Suppose now contrary that $\Gamma_{\rho}(G)>n-m(G)+2$ and let $V_1,\dots,V_t$, $t=\Gamma_{\rho}(G)$, be color classes of a greedy packing $t$-coloring. Let $B_i$ be vertices of $G^i$ that correspond to $V_i$ for every $i\in\{1,\dots,t\}$ and let $B=B_1\cup\cdots\cup B_t$. Now, $V_1$ is a maximal independent set of $G$ and $V_2$ a maximal independent set of $G^2-B_1$ by Lemma \ref{pacgreedycol}. If $V_2=\{x\}$, then either $x$ is not a diametrical vertex of $G$, or if it is, then all vertices that form a diametrical pair with $x$ belong to $V_1$. If $|V_2|>1$, then all the vertices from $V_2$ form pairwise diametrical pairs in $G$. Moreover, $V_2$ forms a maximal clique in $D(G)-V_1$. If $V_2$ in not a maximal clique of minimal cardinality of $D(G)-V_1$, then we get a contradiction with a greedy packing $t$-coloring by replacing $V_2$ with a smaller maximal clique in $D(G)-V_1$. Thus, we have two sets $V_1$ and $V_2$ that fulfills conditions of $m(G)$; however, $|V_1|+|V_2|<m(G)$, a contradiction. Therefore, the equality holds and the proof is completed.  
\end{proof}

The last theorem cannot be improved, as demonstrated by the family $K_{n,n}-M$, $n\geq 3$, where $M=\{u_iv_i:i\in\{1,\dots,n\}\}$ is a perfect matching. Notice that $C_6$ and $Q_3$ belong to this family for $n=3$ and $n=4$, respectively. There are two types of maximal independent sets in $K_{n,n}-M$. The first one is $A_1=\{u_i,v_i\}$ for $i\in\{1,\dots,n\}$ and the second one is $A_2=\{x_1,\dots,x_n\}$ for $x\in\{u,v\}$. A maximal clique of minimum cardinality of $D((K_{n,n}-M))-A_1=(n-1)K_2$ is $K_2$ and of $D((K_{n,n}-M))-A_2=nK_1$ is $K_1$. Hence, $m(K_{n,n})=4$ for $n\geq 3$ and 
\begin{equation}\label{example}
\Gamma_{\rho}(K_{n,n}-M)=2n-2 
\end{equation}
follows by Theorem \ref{diam3}. Besides, notice that for $n>3$ there is no singleton in $D(K_{n,n}-M)-A_1$. Therefore, a simpler general result cannot be obtained for diameter three graphs. Nevertheless, this is possible for many diameter three graphs whenever there exists a singleton in $D(G)-A$, where $A$ is an $i(G)$-set. In such a case, $m(G)=i(G)+1$ and the following result is a direct consequence of Theorem \ref{diam3}: 

\begin{corollary} \label{diam31}
Let $G$ be a graph with ${\rm diam}(G)=3$. If there exists an $i(G)$-set $A$ with a singleton in $D(G)-A$, then $\Gamma_{\rho}(G)=n-i(G)+1$.
\end{corollary}

We can demonstrate the use of Corollary \ref{diam31} on split graphs $S_{p,r}$ of diameter three. Recall that $S_{p,r}$ is a \emph{split graph}  (on $p+r$ vertices) if there exists a partition of $V(S_{p,r})$ into a clique $Q$ on $p$ vertices and into an independent set $I$ on $r$ vertices. If ${\rm diam}(S_{p,r})=3$, then there exists no universal vertex and $p,r\geq 2$. Notice that $S_{p,r}$ can have diameter two, even if there are no universal vertices in $S_{p,r}$. There always exists a vertex $v$ of maximal degree that belongs to $Q$. It is easy to observe that $A=V(S_{p,r})-N(v)$ forms an $i(S_{p,r})$-set of cardinality $p+r-\Delta(S_{p,r})$. Moreover, $u\in Q-\{v\}$ is a singleton in $D(S_{p,r})-A$ and the following result is a direct consequence of Corollary \ref{diam31}:

\begin{corollary} \label{diam35}
If $S_{p,r}$ is a split graph with ${\rm diam}(S_{p,r})=3$, then $\Gamma_{\rho}(S_{p,r})=\Delta(S_{p,r})+1$.
\end{corollary}

Now we briefly shift to the packing chromatic number. Namely, there is no counterpart of Theorem \ref{diam3} for $\chi_{\rho}(G)$ found in the literature. Although symmetric in nature and proved similarly, we present it here for completeness. Recall that $\omega(G)$ is the clique number of a graph $G$, which is the maximum cardinality of a complete subgraph of $G$. 

\begin{theorem} \label{diam32}
Let $G$ be a graph and let $\cal{M}$ be the set of all maximal independent sets of $G$. If ${\rm diam}(G)=3$, then $\chi_{\rho}(G)=|V(G)|-m'(G)+2$, where 
$$m'(G)=\max_{A\in\cal{M}}\{|A|+\omega(D(G)-A)\}.$$
\end{theorem}

\begin{proof}
Let $G$ be a graph with $|V(G)|=n$ and ${\rm diam}(G)=3$. Let $A_0\in \cal{M}$ such that $m'(G)=|A_0|+|Q|$ where $Q$ is a maximum clique of $D(G)-A_0$. Notice that not all vertices of $G$ belong to $A_0$ since ${\rm diam}(G)=3$. Therefore $D(G)-A_0$ contains some vertices and $m'(G)$ is well defined. We construct $B\subseteq V(G(k))$ as follows. All the vertices from $G^1$ that correspond to vertices of $A_0$ in $G$ are put in $B$. Similar we put corresponding vertices from $Q$ in $G^2$ to $B$. In addition $B$ gets exactly one corresponding vertex from $V(G)-(A_0\cup Q)$ for every other $G^j$, $j\geq 3$, until all vertices are used. Clearly, $B$ is an independent set of $G(k)$ and $k=|V(G)|-|A_0|-|Q|+2$ when $|B|=n$. By (\ref{others}) we get $\chi_{\rho}(G)\leq |V(G)|-|A_0|-|Q|+2=|V(G)|-m'(G)+2$. 

Now suppose to the contrary that $\chi_{\rho}(G)<n-m'(G)+2$ and let $V_1,\dots,V_t$, $t=\chi_{\rho}(G)$, be color classes of a greedy packing $t$-coloring. Let $B_i$ be vertices of $G^i$ that correspond to $V_i$ for every $i\in\{1,\dots,t\}$ and let $B=B_1\cup\cdots\cup B_t$. After applying DMP on $B$, we obtain a potentially modified set $B$ with additional properties, namely that $B\cap G^t\neq \emptyset$, because otherwise it would contradict Equation (\ref{others}). Furthermore, vertices of $A_1\subset V(G)$ that correspond to vertices of $B\cap G^1$ form a maximal independent set of $G$. Moreover, vertices of $A_2\subset V(G)$ that correspond to vertices of $B\cap G^2$ form a maximal independent set of $G-A_1$. Since $G$ is diameter three graph, $A_2$ form a maximal clique of $D(G)-A_1$. If $|A_2|\neq\omega(D(G)-A_1)$, then we get a contradiction with a packing $\Gamma_{\rho}(G)$-coloring by replacing $A_2$ with a maximum clique of $D(G)-A_1$. Thus, we have two sets $A_1$ and $A_2$ that fulfills conditions of $m'(G)$; however, $|A_1|+|A_2|\geq |V_1|+|V_2|>m'(G)$, a contradiction. Therefore, the equality holds and the proof is completed.   
\end{proof}

\textcolor{black}{We illustrate also Theorem \ref{diam32} with graph $K_{n,n}-M$ where $M=\{u_iv_i:i\in\{1,\dots,n\}\}$ is a perfect matching of $K_{n,n}$ and $n\geq 3$. As previously mentioned, we have two types of maximal independent sets: either $A_1=\{u_i,v_i\}$ for $i\in\{1,\dots,n\}$ or $A_2=\{x_1,\dots,x_n\}$ for $x\in\{u,v\}$. Clearly, $m'(G)=n+1$ and by Theorem \ref{diam32} we have 
\begin{equation}\label{example1}
\chi_{\rho}(K_{n,n}-M)=n+1.
\end{equation}
}

On the other hand, there is no analogue of Corollary \ref{diam31} for $\chi_{\rho}(G)$. To illustrate this, consider the graph $G$ shown of Figure \ref{G(3)}. The unique $\alpha(G)$-set $A_1$ of $G$ (all degre two vertices and the one in the middle) contains all the diametrical pairs and $D(G)-A_1=6K_1$. Hence, we get $|A_1|+\omega(D(G)-A_1)=5$. The maximal independent set $A_2$ that consists of three black vertices (see graph $G$ of Figure \ref{G(3)}) is not an $\alpha(G)$-set of $G$. In this case, $\omega(D(G)-A_2)=3$ and $m'(G)=6$ is achieved by a maximal independent set that is not an $\alpha(G)$-set.  


\section{Concluding Remarks and Discussion} \label{sec:Conclusion}

A natural relation between packing and Grundy packing chromatic number is the following:
$$\chi_{\rho}(G)\leq \Gamma_{\rho}(G).$$
This raises the natural question of when equality holds and when the difference can be arbitrarily large. It appears that several simple graph families exist where the difference can indeed be arbitrarily large. First, such a family forms stars $S_n=K_{n-1,1}$. We have $\Gamma_{\rho}(S_n)=n$ by Corollary \ref{join} and it is easy to see that $\chi_{\rho}(S_n)=2$. Thus, we have a family of graphs where the difference $\Gamma_{\rho}(S_n)-\chi_{\rho}(S_n)=n-2$ tends to infinity as $n$ grows.    

\textcolor{black}{Another family with arbitrary large difference between $\Gamma_{\rho}(G)$ and $\chi_{\rho}(G)$ represents graphs $K_{n,n}-M$ for $n\geq 3$, where $M$ is a perfect matching. Recall that $\Gamma_{\rho}(G_n)=2n-2$ by (\ref{example}) and $\chi_{\rho}(G_n)=n+1$ by (\ref{example1}).} So, again $\Gamma_{\rho}(G_n)-\chi_{\rho}(G_n)=n-3$ tends to infinity when $n$ grows. The next result now follows directly.  

\begin{proposition} \label{large}
For an arbitrary positive integer $k$, there exists a graph $G$ such that $\Gamma_{\rho}(G)-\chi_{\rho}(G)\geq k$.
\end{proposition}

A natural family where the equality between $\Gamma_{\rho}(G)$ and $\chi_{\rho}(G)$ holds are complete graphs. Besides, we do not expect this equality to hold for many graphs. If we restrict ourselves to diameter two graphs, then by (\ref{diam2}) and Proposition \ref{diam22}, $\alpha(G)=i(G)$ holds. 
A graph is said to be \emph{well-covered} if all its (inclusion-)maximal independent sets have the same size, that is when $\alpha(G)=i(G)$. Well-covered graphs were introduced by Plummer in 1970 \cite{Plummer1970} and have been studied extensively in the literature (see, e.g., the survey papers \cite{Plummer1993, Hartnell13}). Whence the next result and following general problem.  

\begin{corollary} \label{wellcov}
If $G$ is well-covered graph with ${\rm diam}(G)=2$, then $\Gamma_{\rho}(G)=\chi_{\rho}(G)$.
\end{corollary}

\begin{problem} \label{prob1}
Describe all graphs for which $\Gamma_{\rho}(G)=\chi_{\rho}(G)$ holds. In particular, does this equality hold for all well-covered graphs?
\end{problem}

For diameter three graphs, we pose two natural problems. The first one is with respect to Corollary \ref{diam31}, where we are interested in all graphs for which the equality from Corollary \ref{diam31} holds. To this end,  we call a graph $G$ \emph{diametrical} if every vertex of $G$ is a diametrical vertex. 

\begin{problem} \label{prob2}
Describe all diameter three graphs for which $\Gamma_{\rho}(G)=|V(G)|-i(G)+1$ holds. In particular, does any diametrical graph satisfy the equality and are there any exceptions to the equality besides diametrical graphs?
\end{problem}

The second problem we pose is motivated by the example described after Theorem \ref{diam32}, where  $\chi_{\rho}(G)$-coloring does not result in the vertices of color one forming an $\alpha(G)$-set. We have not found any example where the vertices of color one do not form an $i(G)$-set in a $\Gamma_{\rho}(G)$-coloring.   

\begin{problem} \label{prob3}
Describe all diameter-three graphs where the vertices of color one do not form an $i(G)$-set for any $\Gamma_{\rho}(G)$-coloring. 
\end{problem}

We conclude with $\Gamma_{\rho}(P_k)$. This is a first step into the area of infinite lattices, which have been extensively studied for the packing chromatic number as mentioned in the introduction. Clearly, $\Gamma_{\rho}(\mathbb{Z}^3)$ is not finite, since $\chi_{\rho}(\mathbb{Z}^3)$ is not finite (see \cite{finbow-rall}). However, the questions concerning $\Gamma_{\rho}(\mathbb{Z}^2)$, $\Gamma_{\rho}({\cal H})$ for a hexagonal lattice ${\cal H}$, and $\Gamma_{\rho}(P_k \Box \mathbb{Z})$, where $\Box$ denotes the Cartesian product, remain open. We refrain from making predictions regarding the finiteness of the Grundy packing chromatic number for the aforementioned (and other) infinite lattices. We begin with a path, where $\Gamma_{\rho}(P_k)$ is bounded, extending from the following result for finite paths to a (two-way) infinite path. 

\begin{theorem} 
\textcolor{black}{We have $\Gamma_{\rho}(\mathbb{Z})=7$. More precisely, the following holds:
\begin{equation*}
\Gamma_{\rho}(P_k)=\left\{ 
\begin{array}{ccc}
k & : & k\leq 3, \\
k-1 & : & 4\leq k\leq 5, \\
4 & : & 6\leq k\leq 7, \\
5 & : & 8\leq k\leq 14, \\
6 & : & 15\leq k\leq 24, \\
7 & : & k\geq 25
\end{array}
\right. .
\end{equation*}}
\end{theorem}

\begin{proof}
\textcolor{black}{Notice that $\Gamma_{\rho}(\mathbb{Z})=7$ follows from the result for paths and therefore, we concentrate on paths. The result follows for $k\leq 3$ by Proposition \ref{ncol} and for $4\leq k\leq 5$ by Theorem \ref{n-1}. For the remainder, we first discuss the lower bound for $\Gamma_{\rho}(P_k)$. Observe that the packing colorings presented in Figure \ref{paths} represent the desired colorings for $P_4,P_5,P_8$ and $P_{15}$, respectively. Besides,} it is straightforward to verify that the following coloring is indeed a greedy packing $7$-coloring of $P_{25}$:
$$3124152136124137121531241.$$
\textcolor{black}{Now let $\ell$ be a positive integer different from $A=\{1,2,3,4,5,8,15,25\}$. It is easy to enlarge the given packing coloring for $P_i$, $i\in A$, to a packing coloring of $P_{\ell}$, where $i$ is the greatest element of $A$ smaller than $\ell$. Thus, the lower bound follows.} 

\textcolor{black}{For the other inequality we first show that if $\Gamma_{\rho}(P_k)\geq 5$, then $k\geq 8$, which settles the case for $5\leq k\leq 6$. For this let $c(x)=5$ and $c(y)=4$ where $d(x,y)\leq 4$. If $xy\in E(P_k)$, then the coloring $1321451231$ is enforced, and $k\geq 10$. If $d(x,y)=2$, then their common neighbor must have color $1$. To have a packing coloring we need to have colors $213$ (or even a longer sequence) on the other sides of $x$ and $y$. Hence, we get $312415213$, which means that $k\geq 9$. Let now $d(x,y)=3$ with the (potential) sequence of colors $abcd4ef5ghi$. If $f=1$, then $e\neq 1$ and $d=1$. This also means that $e\neq 3$ and $e=2$ follows. Observe that $\{a,b,c\}=\{1,2,3\}$ to ensure that color $3$ is close enough to $y$, which already yields $k\geq 8$. Now let $e=1$. Then $f\neq 1$ and $g=1$. This also means that $f=2$. Now, $\{b,c,d\}=\{1,2,3\}$ to ensure that color $3$ is close enough to $y$, which again gives $k\geq 8$. It remains to check the case where $d(x,y)=4$. If there are two vertices of color $1$ between $x$ and $y$, then we have a sequence $41215abc$. Again $\{a,b,c\}=\{1,2,3\}$ to ensure that color $3$ is close enough to $x$ and therefore, $k\geq 8$. If there is only one vertex between $x$ and $y$ with color $1$, then we have two shortest possibilities $14213512$ or $2143125132$ (if neighbors of $x$ and $y$ have the same color, the sequence is even longer). In all cases we have $k\geq 8$ and we are done with $5\leq k\leq 6$.}

\textcolor{black}{For the remainder, we first demonstrate that two consecutive vertices cannot both be assigned colors greater than 4. Suppose to the contrary that colors $i\geq 5$ and $j>i$ are on two consecutive vertices $x$ and $y$, respectively. It is straightforward to observe that four vertices preceding $x$ and following $y$ must receive colors $1,2,3,1$. However, then $x$ and $y$ have no vertex of color $4$ at distance at most $4$ from them, a contradiction with a greedy packing $\Gamma_{\rho}(P_k)$-coloring.
}\\

\textcolor{black}{ Next, we show that in a packing coloring of a path, if two vertices with colors $5$ and color $a\in\{6,7\}$ are at distance at most five from each other, then
the path connecting them does not include a vertex with color $4$. Suppose, for the sake of contradiction, that this is not the case. The sequences of colors $2154121a21$, $51214a12$, $215412a12$ and $15214a12$ show that $4$ is not a neighbor of $5$ or $a$ because vertex of color $4$ has no vertex of color $3$ at distance at most $3$. If it is four, then we have a sequence $5141a$ and $4$ has no color $2$ close enough, so the distance must be five. In this case we have $52141a31$, $1351412a1$, $51241a31$ or $1351421a$ and either $a$ or $5$ has no color $2$ close enough. Hence a contradiction in all cases.
} 

\textcolor{black}{To prove that $\Gamma_{p}(P_k) = 5$ if $8 \le k \le 14$, we assume that $\Gamma_{\rho}(P_k)\geq 6$ now and we will show that $k\geq 15$ in this case. Let $x$ and $y$ be vertices colored with $6$ and $5$, respectively, with $2\leq d(x,y)\leq 5$.  We systematically examine different sequences of packing colorings of a path depending on the distance between $x$ and $y$ and all possible colors of vertices between them (where color $4$ is not among them, as shown earlier). We will omit several details, but usually positions of colors $1,2$ and $3$ are forced in these sequences to ensure the feasibility of a packing coloring. Every contradiction is marked by $a/b$, which means that color $a$ does not have color $b$ close enough.  If $d(x,y)=2$, then we get the following minimum sequence of colors on $P_{15}$:
$$214312516213412.$$
If $d(x,y)=3$, then we have the following minimal forced possible sequences of colors:}

\begin{center}
\begin{tabular}{|c|c|}\hline
$512613214$ & a $6/4$ contradiction\\ \hline
$512614312$ & a $4/2$ contradiction \\ \hline
$123152164123$ & a $6/3$ contradiction \\ \hline
$51361$ & a $3/2$ contradiction\\ \hline
$153162$ & a $3/2$ contradiction\\ \hline
\end{tabular}%
\end{center}

\begin{center}
\begin{tabular}{|c|c|}\hline
$31241521361241321$ & $P_{17}$\\ \hline
$13214215312614213$ & $P_{17}$ \\ \hline
$52126$ & a contradiction with color $2$\\ \hline
$1531361$ & a $3/2$ contradiction\\ \hline
$142135121631241$ & $P_{15}$\\ \hline
$132145121641231$ & a $6/3$ and $5/3$ contradiction \\ \hline
$51316$ & a contradiction with $3/2$\\ \hline 
\end{tabular}%
\end{center}
\textcolor{black}{If $d(x,y)=5$, we have the following eight minimal forced possibilities presented in the following table:}\\
\begin{center}
\begin{tabular}{|c|c|}\hline
$21341251321641231$ & $P_{17}$ \\ \hline
$13214512316214312$ & $P_{17}$ \\ \hline
$51212613214$ & a $6/4$ contradiction\\ \hline
$142135121361241321$ & $P_{17}$ \\ \hline
$21431251312614213$ & $P_{17}$ \\ \hline
$41231521216$ & a $5/4$ contradiction\\ \hline
$31241521316214312$ & $P_{17}$\\ \hline
$123142153121631241$ & $P_{18}$\\ \hline
\end{tabular}%
\end{center}

\textcolor{black}{To prove that $\Gamma_{p}(P_k)=6$ if $15 \le k \le 24$, we assume that $\Gamma_{\rho}(P_k)\geq 7$ and we will show that $k\geq 25$ in this case. We continue in a similar fashion as before. Let $x,y,z$ and $w$ be vertices colored by $7,6,5$ and $4$, respectively, with $2\leq d(x,y)\leq 6$, $2\leq d(x,z)\leq 5$ and $d(x,w)\leq 4$. Further, recall that vertex of color $4$ is not between $x$ and $z$. First, we show that vertex colored with $4$ must be between $x$ and $y$. It is easy to see from the sequences of colors $715214$ (distance between $x$ and $w$ is too big) and $7215412$ (a $5/3$ contradiction), that $z$ cannot be between $x$ and $w$. Similarly, we show that $y$ is not between $x$ and $w$. This is not possible when $d(x,y)=2$ by the sequence of colors $12617$. If $d(x,y)=3$, then a possible color $3$ between $x$ and $y$ yields a $3/2$ contradiction. We are left with sequences $2146127$ (a $6/3$ contradiction) and $46217$ (a $6/1$ contradiction). So, $w$ is between $x$ and $y$. 
}

\textcolor{black}{Now we analyze different sequences of colors that depend on $d(x,y)$. Clearly, $d(x,y)>2$ because $647$ gives $4/1$ contradiction. If $d(x,y)=3$, then the only possible sequences $61471$ and $16417$ yield a $4/2$ contradiction. Now let $d(x,y)=4$. All possible sequences in the following table result in a contradiction:}
\begin{center}
\begin{tabular}{|c|c|}\hline
$31641271$ & a $6/2$ contradiction \\ \hline
$164137$ & a $4/2$ contradiction\\ \hline
$61417$ & a $4/2$ contradiction\\ \hline
$6214713$ & a $7/2$ contradiction\\ \hline
$631471$ & a $4/2$ contradiction\\ \hline 
\end{tabular}%
\end{center}
\textcolor{black}{Moreover, if $d(x,y)=5$, all the possible color sequences give a contradiction as shown in the next table. Here, notice that between $x$ and $y$ is $w$ together with two vertices of color $1$ and one vertex of color $2$-color $3$ gives an immediate $3/2$ contradiction.}
\begin{center}
\begin{tabular}{|c|c|}\hline
$21641217$ & a $4/3$ contradiction \\ \hline
$61214712$ & a $4/3$ contradiction\\ \hline
$13614217$ & a $6/2$ contradiction\\ \hline
$61241731$ & a $7/2$ contradiction\\ \hline
$136141271$ & a $6/2$ contradiction\\ \hline
$62141731$ & a $7/2$ contradiction\\ \hline
\end{tabular}%
\end{center}
\textcolor{black}{We finish this case with $d(x,y)=6$. Here, it is easy to see that we get a contradiction if there are three vertices of color one between $x$ and $y$ by sequences $612141731$-a $7/2$ contradiction, and $136141217$-a $6/2$ contradiction. So, we have between $x$ and $y$ two vertices of color $1$, $w$ and either two vertices of color $2$ or one vertex of color $2$ and one vertex of color $3$ (notice that if we have two vertices of color $3$ an immediate $3/2$ contradiction follows). We analyze the different positions of the mentioned vertices between $x$ and $y$ in the following table. Notice here that sometimes endings are not really fixed, but any change of them gives the same outcome.}
\begin{center}
\begin{tabular}{|c|c|}\hline
$6214217$ & a $4/3$ contradiction\\ \hline
$6212417$ & a contradiction with color $2$\\ \hline
$216314217$ & a $3/2$ contradiction \\ \hline
$6214317$ & a $3/2$ contradiction\\ \hline
$142135121631241721351241321$ & $P_{27}$ \\ \hline
$64123172134125$ & a $7/5$ contradiction \\ \hline
$16413217$ & a $4/2$ contradiction\\ \hline
$31541261321471$ & a $5/2$ contradiction\\ \hline
$3124152136124137121531241$ & $P_{25}$\\ \hline
$6134127$ & a $3/2$ contradiction\\ \hline
$213412513216142137121531241$ & a $P_{27}$\\ \hline
$6143127$ & a $4/2$ contradiction \\ \hline
$316412137$ & a $6/2$ contradiction\\ \hline
$631214713$ & a $7/2$ contradiction\\ \hline
$164131271$ & a $4/2$ contradiction\\ \hline
$62131471$ & a $4/2$ contradiction \\ \hline
$162141371$ & a $3/2$ contradiction\\ \hline
$163141271$ & a $3/2$ contradiction\\ \hline
\end{tabular}%
\end{center}

Finally, we assume that $\Gamma_{\rho}(P_k)>7$ for some $k$ and let $c=(V_1,\dots,V_{\Gamma_{\rho}(P_k)})$ be a greedy packing $\Gamma_{\rho}(P_k)$-coloring of $P_k$. Since $\Gamma_{\rho}(P_k)>7$, there exists a vertex $v$ with $c(v)=8$. Let $x_i$ and $y_i$ be vertices at distance $i$ from $v$ for any possible $i$. We may assume that $x_i$ precedes $v$ and $y_i$ follows $v$. \textcolor{black}{By symmetry,} we may assume that all the colors $1,\dots,7$ must appear on vertices $x_1,\dots,x_7,y_1\dots,y_6$ and let $P$ be a $x_7,y_6$-subpath of $P_k$. We separate three cases with respect to the positions of vertices of color $1$ that are adjacent to $v$. In all three cases we will obtain a contradiction, which shows that $\Gamma_{\rho}(P_k)\leq 7$ as well.

\noindent \textbf{Case 1.} $c(y_1)=1\neq c(x_1)$.  
In this case, $c(x_2)=1$ and the $i(P_k)$-set that contains $y_1$ and $x_2$ \textcolor{black}{contains} also $y_4,y_7,x_5,x_8$. Optimal positions for color two are vertices $x_1,x_6,y_3,y_8$ to dominate $P$ to obtain a maximal independent set of $G^2-V_1$. Now, $P^3-(V_1\cup V_2)$ is a path $x_7x_4x_3vy_2y_5y_6$ and at least two vertices of them must have color $3$ for a maximal independent set (when $x_7$ or $y_6$ has a vertex of color $3$ at distance at most $3$ outside of $P$). Therefore, exactly four vertices remain for colors $4,5,6,7$. Clearly, $c(x_7)=7$ and $c(y_6)=6$. Moreover, two consecutive vertices cannot have colors $5$ and $6$, which yields $c(y_5)\neq 5$ and $c(y_5)\neq 4$ because $d(v,y_5)=5$. Hence, $c(y_5)=3$ and consequently $c(x_3)=3$. If $c(y_2)=4$, then $c(x_4)=5$, a contradiction because $x_4$ has no vertex of color $4$ at distance at most $4$ \textcolor{black}{(here $c(x_8)=1$, otherwise $x_7$ has no neighbors of color one)}. Similarly, if $c(x_4)=4$, then $c(y_2)=5$, a contradiction since $y_2$ has no vertex of color $4$ at distance at most $4$. So, the coloring is not optimal for the first two color classes (we have more vertices of color $1$ or $2$ than needed). Nevertheless, now we have more vertices of color $1$ or $2$ and we are left with only three vertices for the remaining four colors $4,5,6,7$, which is again not possible.     

\noindent \textbf{Case 2.} $c(x_1)=1\neq c(y_1)$. We have $c(y_2)=1$ and in optimal case $c(y_5)=1$ and either $c(y_7)=1$ or $c(y_8)=1$. Besides, we have at least two vertices of color $1$ among $x_i$, $I\in\{3,\dots,7\}$. Hence, we have at least five vertices of color $1$ on $P$. Now, $P^2-V_1=P_9$ and there are at least three vertices in $V_2$. Further, $P^3-(V_1\cup V_2)=P_6$ and there are at least two vertices of $P$ in $V_3$. With this, only three vertices remain for four colors $4,5,6,7$, which is not possible. If the coloring is not optimal, then we have even less vertices for colors $4,5,6,7$, a contradiction.  

\noindent \textbf{Case 3.} $c(y_1)=1=c(x_1)$. In optimal case $c(y_4)=1=c(y_7)$ and there are two vertices of color $1$ among $x_i$, $I\in\{3,\dots,7\}$. So, at least five vertices of color $1$ are on $P$. Now, $P^2-V_1=P_9$ and there are at least three vertices in $V_2$. Furthermore, $P^3-(V_1\cup V_2)$ contains six vertices which induce $P_6$ or $P_5\cup K_1$, and there are at least two vertices of $P$ in $V_3$. With this, only three vertices remain for four colors $4,5,6,7$, which is not possible. If the coloring is not optimal, then we have even less vertices for colors $4,5,6,7$, a contradiction.  
\end{proof}

\end{document}